\documentclass[12pt]{amsart}

\usepackage[hmargin=1in,vmargin=1in]{geometry}

\usepackage{amssymb,amsmath,color}
\usepackage[dvipsnames]{xcolor}
\usepackage{amsthm}
\usepackage{url}
\usepackage{tikz}
\usepackage{hyperref}
\usepackage{comment}
\usepackage{enumitem}

\definecolor{red}{rgb}{1,0,0}
\definecolor{blue}{rgb}{.2,.2,.8}

\newtheorem{theorem}{Theorem}[section]
\newtheorem{corollary}[theorem]{Corollary}

\newtheorem{conjecture}[theorem]{Conjecture}
\newtheorem{lemma}[theorem]{Lemma}

\theoremstyle{definition}
\newtheorem{definition}[theorem]{Definition}
\newtheorem{example}[theorem]{Example}
\newtheorem{remark}[theorem]{Remark}

\def\la{\lambda}
\def\cP{\mathcal P}

\def\cB{\mathcal B}
\def\cS{\mathcal S}
\def\cD{\mathcal D}
\def\cO{\mathcal O}
\def\cR{\mathcal R}
\def\cC{\mathcal C}

\newcommand{\injclass}{\cS_k}
\newcommand{\injtwo}{\cS_2}
\def\part{\mathfrak{p}}
\newcommand{\colorclass}[2]{\cC_{#1}(#2)}
\newcommand{\colorclasscount}[2]{c_{#1}(#2)}
\newcommand{\twocolorclass}[3]{\cC_{#1}^{#2}(#3)}
\newcommand{\twocolorclasscount}[3]{c_{#1}^{#2}(#3)}
\newcommand{\mult}[2]{m_{#1}(#2)}

\DeclareMathOperator{\pre}{pre}

\begin{document}

\title{On partitions associated with elementary symmetric polynomials}

\begin{abstract} 
The elementary symmetric partition function is a map on the set of partitions. It sends a partition $\lambda$ to the partition whose parts are the summands in the evaluation of the elementary symmetric function on the parts of $\lambda$. These elementary symmetric partition functions have been studied before, and are related to plethysm. In this note, we study properties of the elementary symmetric partition functions, particularly related to injectivity and the number of parts appearing in their image partitions. 
\end{abstract}

\author[C. Ballantine]{Cristina  Ballantine}
\address{Department of Mathematics and Computer Science\\ College of the Holy Cross \\ Worcester, MA 01610, USA} 
\email{cballant@holycross.edu} 

\author[S. Nazir]{Shaheen Nazir}
\address{Mathematics Department \\ Syed Babar Ali School of Science and Engineering \\ Lahore University of Management Sciences\\ Lahore,  Pakistan}
\email{shaheen.nazir@lums.edu.pk}

\author[B.~E.~Tenner]{Bridget Eileen Tenner}
\address{Department of Mathematical Sciences, DePaul University, Chicago, IL, USA}
\email{bridget@math.depaul.edu}

\author[K. Westrem]{Karlee Westrem}
\address{Department of Mathematical Sciences \\ Appalachian State University\\ Boone, NC, USA}
\email{westremk@appstate.edu}

\author[C. Zhao]{Chenchen Zhao}
\address{Department of Mathematics \\ University of California, Davis\\ Davis, CA, USA}
\email{cczhao@ucdavis.edu}

\maketitle

\allowdisplaybreaks

\section{Introduction}\label{sec:intro}

A \emph{partition} of a positive integer $n$ is a weakly decreasing sequence $\lambda=(\lambda_1,\ldots,\lambda_\ell)$ of positive integers whose sum is $n$.   The numbers $\lambda_i$ are called the {\em parts} of $\lambda$. In~\cite{bbm}, Ballantine, Beck, and Merca introduced partitions associated with elementary symmetric polynomials. Recall that the {\em $k$th elementary symmetric polynomial} in a set of variables 
$\{x_1,\ldots,x_\ell\}$  is
$$e_k(x_1,\ldots,x_\ell) = \sum_{i_1<\cdots< i_k} x_{i_1}\cdots x_{i_k}.$$
For  example,
$$
e_2(x_1,x_2,x_3,x_4) = x_1 x_2 + x_1 x_3 + x_1 x_4 + x_2 x_3 + x_2 x_4 + x_3 x_4.
$$

Our work here studies a function on partitions.

\begin{definition}\label{defn:prek}
Given a partition $\lambda=(\lambda_1,\ldots,\lambda_\ell)$  and $k \le \ell$, the \emph{$k$th elementary symmetric partition}
$\pre_k(\la)$
is the partition whose parts are the summands in the evaluation $e_k(\la_1,\ldots,\la_\ell)$. 
When $\ell < k$, we say that $\pre_k(\lambda)$ is undefined.
\end{definition}

Note that \cite{bbms} used the convention that $\pre_k$ returned $0$ on partitions with fewer than $k$ parts. We change that convention here in order to simplify notation. This change does not meaningfully affect any of our work or results, because we are only interested in when evaluating $\pre_k$ ``makes sense.''

\begin{example}
For $\lambda=(4,3,2,2)$ and $k = 2$, we have
$$
e_2(4,3,2,2) = 4\cdot 3 + 4\cdot 2 + 4\cdot 2 + 3\cdot 2 + 3\cdot 2 + 2\cdot 2
$$
and
$$
\pre_2(4,3,2,2) = (12, 8, 8, 6, 6, 4).
$$
\end{example}

Properties of a particular family of elementary symmetric partitions for $k=2$ were studied in~\cite{bbms}. In particular, that work derived identities involving the sequences that count the parts of a given value in the image of the set of binary partitions of $n$ under $\pre_2$, where a binary partition is a partition whose parts are all powers of $2$. Here, we expand on that work. 

Given a partition $\lambda=(\lambda_1,\ldots,\lambda_\ell)$, we call $\sum_{i=1}^{\ell} \lambda_i$ the \emph{size} of $\lambda$, and $\ell$ the \emph{length} of $\lambda$. 
Although partition parts must be positive, we also define the \emph{empty partition} to be the (unique) partition of $0$. We use the following notation:
\begin{align*}
\cP(n)&=\{\la \mid  \text{$\la$ is a partition of $n$}\},\\
\part(n) &= |\cP(n)|, \text{ and}\\
\cP&=\bigcup_{n\ge0} \cP(n).
\end{align*}
For example, $\part(5)=7$ because 
$$\cP(5)=\{(5),\,  (4,1), \, (3,2), \, (3,1,1), \, (2,2,1), \, (2,1,1,1), \, (1,1,1,1,1)\}.$$

The function $\pre_k$ is only defined on partitions having at least $k$ parts, we will abuse notation slightly for ease and clarity, and refer to $\pre_k : \cP \rightarrow \cP$ or $\pre_k : \cP(n) \rightarrow \cP$, with the understanding that $\pre_k$ is only defined on the sensible subset of the relevant ``domain.''

The following is the most consequential conjecture of~\cite{bbm}.

\begin{conjecture}[\cite{bbm}]
\label{prek:con}
For any $k\ge1$ and $n\ge0$, the map $\pre_k: 
\cP(n)\to\cP$ is injective.   
\end{conjecture}

Certainly Conjecture~\ref{prek:con} is trivially true when $k=1$, because $\pre_1$ is the identity map.

The following weaker form of the conjecture was proved in~\cite{bbm}, where $\cB(n)=\{\la \mid  \text{$\la$ is a binary partition of $n$}\}$, and $\cB$ 
is defined analogously.  

\begin{theorem}[\cite{bbm}] \label{injb2}
For any $n\ge0$, the map  
$\pre_2:\cB(n)\to \cB$ is injective. 
\end{theorem}

In the proof of Theorem~\ref{injb2}, it is shown that if $\lambda\in 
\cB(n)$ and $\nu=\pre_2(\lambda)$, then the multiplicities of parts in $\nu$ determine (recursively) the multiplicities of parts in $\lambda$, thus uniquely defining $\lambda$.

The case $k=2$ of Conjecture~\ref{prek:con} was recently, and in parallel with our work,  proved in~\cite{Li}. 
In the current article, we prove the injectivity of $\pre_k$ on a particular subset of 
$\cP$ for all $k\geq 2$. 
We also prove several conjectures  from~\cite{bbms} and from private correspondence with George Beck that rely only on this restricted injectivity result.  Our proof of the injectivity of $\pre_k$, while only for a subset of 
$\cP(n)$, is very different from that in~\cite{Li}. We are able to completely recover the partition $\lambda$ from its image $\nu=\pre_k(\lambda)$, and the result does not require a fixed value of $n$.

The paper is organized as follows. In Section~\ref{sec:on injectivity of pre_k}, we establish injectivity of $\pre_k$ on a large class of partitions. That result, Theorem~\ref{th:injectivity}, proves injectivity on a class that is neither a subset nor a superset of $\cP(n)$for any fixed $n$, and so our result is related to but distinct from Conjecture~\ref{prek:con}. Section~\ref{sec:consequences of our injectivity} uses the injectivity established in Theorem~\ref{th:injectivity} to extend various results from~\cite{bbm,bbms}. In Section~\ref{sec:consequences of full injectivity}, we discuss how Li's   recent result \cite{Li} relates to plethysm, and we conclude the paper in Section~\ref{sec:open} with a selection of directions for future research.

\section{On injectivity of the elementary symmetric partition}\label{sec:on injectivity of pre_k}

In this section, we prove injectivity of $\pre_k$ on a particular subset of $\cP$, which intersects  
$\cP(n)$ nontrivially. In this way, our result is simultaneously not fully addressing Conjecture~\ref{prek:con}, while also addressing a larger question.

We begin with some terminology and notation.

\begin{definition}\label{defn:multiplicity}
    For a partition $\lambda$ and a positive integer $i$, the \emph{multiplicity} of $i$ in $\lambda$, denoted 
    $$\mult{\lambda}{i},$$
    is the number of times that $i$ occurs as a part in $\lambda$. 
\end{definition}

The subset of $\cP$ on which we can show injectivity of $\pre_k$ is defined by the multiplicity of certain values.

\begin{definition}
    Let $\injclass$ be the set of partitions $\lambda$ for which 
    $$\begin{cases}\mult{\lambda}{1} \ge k,\ \text{or}\\
    \mult{\lambda}{1} = k-1 \text{ and } \mult{\lambda}{p} \ge 1 \text{ for some prime } p.
    \end{cases}$$
\end{definition}

We will prove injectivity of $\pre_k$ on $\injclass$. Note that we put no restrictions on the size of the partitions in $\injclass$, meaning that each $\injclass$ is a subset of $\cup_n \cP(n)$. Therefore, on $\injclass$, the result achieved is more general than what had been hoped for in Conjecture~\ref{prek:con}.

\begin{theorem}\label{th:injectivity} 
The map $\pre_k:\injclass\to \cP$ is injective. 
    \end{theorem}

The proof of Theorem~\ref{th:injectivity} follows from the following two  lemmas.

\begin{lemma}\label{lem:injectivity on partitions with 1s}
    Let $k\geq 2$ and let $\nu$ be a partition in the image of  $\pre_k$. If $\mult{\nu}{1} > 0$, then $\nu$ has a unique preimage $\lambda := \pre_k^{-1}(\nu)$.
\end{lemma}

\begin{proof}
    Fix some $\lambda$ for which $\pre_k(\lambda) = \nu$. We will show that the multiplicity data $\{\mult{\nu}{i}\}$ determines the multiplicity data $\{\mult{\lambda}{i}\}$, and hence $\nu$ determines $\lambda$ uniquely.

    The only way to obtain a part of $\nu = \pre_k(\lambda)$ equal to $1$ would be from $k$ parts equal to $1$ in $\lambda$, namely $\lambda_{i_1} = \cdots =\lambda_{i_k} = 1$, with $i_1 > \cdots > i_k$. Thus
    \begin{equation}\label{eqn:binomial for counting 1s}
        \binom{\mult{\lambda}{1}}{k} = \mult{\nu}{1}.
    \end{equation}
    In particular, since the sequence $\left\{\binom{n}{k}\right\}_{n\geq k}$, is strictly increasing,   Equation~\eqref{eqn:binomial for counting 1s} has a unique solution when $\mult{\nu}{1} > 0$, and thus $\mult{\lambda}{1}$ is determined by $\nu$.

     Next, we use similar reasoning to recover $\mult{\lambda}{p}$ for any prime $p$. 
 The only way to obtain a part of $\nu = \pre_2(\lambda)$ equal to $p$ would be from $k-1$ parts of $\lambda$ equal to $1$ and a single part of $\lambda$ equal to $p$. Therefore
    $$\mult{\nu}{p} = \binom{\mult{\lambda}{1}}{k-1} \cdot \mult{\lambda}{p}.$$
    Therefore, 
    $$\label{eqn:counting primes}
        \mult{\lambda}{p} = \frac{\mult{\nu}{p}}{\binom{\mult{\lambda}{1}}{k-1}},$$
    and hence, the multiplicity $\mult{\lambda}{p}$ is determined by  $\nu$.

    Now fix a composite number $y$ and assume, inductively, that for all $x < y$, we can compute $\mult{\lambda}{x}$ solely from the partition $\nu$. 
 Let $D_y$ be the collection of all  tuples $(d_1^{b_1},\ldots, d_t^{b_t})$,  satisfying 
 \begin{itemize}
 \item  $1<d_1<\cdots<d_t<y$, 
 \item $b_1,  \ldots, b_t>0$, 
 \item $b_1+ \cdots+ b_t\leq k$, and
 \item $d_1^{b_1}\cdots d_t^{b_t}=y$.
 \end{itemize}
 Then, 
     \begin{equation}\label{eqn:solving for m(y)}
         \mult{\nu}{y}=\binom{\mult{\lambda}{1}}{k-1}\mult{\lambda}{y}+ \hspace{-.1in}\sum_{(d_1^{b_1},\ldots, d_t^{b_t})\in D_y}\hspace{-.1in} \binom{\mult{\lambda}{1}}{k-(b_1+\cdots+b_t)} \binom{\mult{\lambda}{d_1}}{b_1}\cdots \binom{\mult{\lambda}{d_{t}}}{b_t}.
    \end{equation}
We can rearrange Equation~\eqref{eqn:solving for m(y)} to compute $\mult{\lambda}{y}$ from $\{\mult{\lambda}{d_{i}}\}$, $\mult{\lambda}{1}$, and $\mult{\nu}{y}$. And by induction, all of these can be recovered from $\nu$. Therefore, for all $i$, we can compute $\{\mult{\lambda}{t}\}$ entirely from $\{\mult{\nu}{t}\}$. Hence $\lambda$ is determined uniquely by $\nu$. 
\end{proof}

We demonstrate the computations of Lemma~\ref{lem:injectivity on partitions with 1s} with an example.

\begin{example} 
Let $\nu = \pre_3(\lambda)$ for $\lambda = (72, 8, 6,6, 6, 4, 3, 3, 2, 1, 1, 1)$. We want to show that $\lambda$ is the only preimage of $\nu$ under $\pre_3$. As defined in the proof of Lemma~\ref{lem:injectivity on partitions with 1s},
 \begin{align*} 
 D_{72} = \left\{(2^2,18),\!\right.& \left. (2,3,12),\,(2,4,9), \,(2,6^2), \,(2,36),  \right.\\ & \left.  (3^2,8),\, (3,4,6), \, (3,24),\, (4,18),\,  (6,12), \, (8,9)\right\}.
 \end{align*} 
 Note that $\mult{\nu}{72} = 13$, and $\mult\lambda{1} = 3$. Among all tuples in $D_{72}$, only three actually contribute to the sum in Equation~\eqref{eqn:solving for m(y)}:
 \begin{align*}
      \sum_{(d_1^{b_1},\ldots, d_t^{b_t})\in D_{72}}& \binom{\mult{\lambda}{1}}{k-(b_1+\cdots+b_t)} \binom{\mult{\lambda}{d_1}}{b_1}\cdots \binom{\mult{\lambda}{d_{t}}}{b_t}\\ 
      = & \ \binom{\mult{\lambda}{2}}{1} \binom{\mult{\lambda}{6}}{2} + \binom{\mult{\lambda}{3}}{2} \binom{\mult{\lambda}{8}}{1} + \binom{\mult{\lambda}{3}}{1}\binom{\mult{\lambda}{4}}{1} \binom{\mult{\lambda}{6}}{1}\\
      = & \ 1\cdot 3+1\cdot 1+2\cdot 1\cdot 3 = 10.
 \end{align*}
Thus Equation~\eqref{eqn:solving for m(y)} gives $13 = \binom{3}{3-1}\mult{\lambda}{72} +10$, yielding $\mult{\lambda}{72} = 1$, as desired.
\end{example}

While Lemma~\ref{lem:injectivity on partitions with 1s} shows that $\pre_k$ is injective onto the partitions with at least one part equal to $1$ (and hence from the elements of $\injclass$ with $\mult{\lambda}{1} \ge k$), the other component to proving Theorem~\ref{th:injectivity} considers partitions in the image of $\pre_k$ that have prime parts.

\begin{lemma}\label{lem:injectivity on partitions with no 1s but a prime part}
    Let $k\geq 2$ and $\nu$ be a partition in the image of $\pre_k$. If $\mult{\nu}{p} > 0$ for some prime $p$, then $\nu$ has a unique preimage $\lambda := \pre_k^{-1}(\nu)$.
\end{lemma}

\begin{proof}
    Fix some $\lambda$ for which $\pre_k(\lambda) = \nu$. As in the proof of Lemma~\ref{lem:injectivity on partitions with 1s}, we will show that $\lambda$ is unique with this property. If $\mult{\nu}{1} > 0$, then Lemma~\ref{lem:injectivity on partitions with 1s} recovers $\lambda$. Therefore, assume, for the remainder of the proof, that $\mult{\nu}{1} = 0$.
    
    The only way to obtain a part of $\nu = \pre_k(\lambda)$ equal to $p$ would be from $k-1$ different parts of $\lambda$ equal to $1$ and one part equal to $p$. Thus
    $$\mult{\nu}{p} = \binom{\mult{\lambda}{1}}{k-1}\cdot \mult{\lambda}{p}.$$ 
    Since $\mult{\nu}{p} > 0$, we need $\mult{\lambda}{1} \geq k-1$. Since $\mult{\nu}{1}=0$, we have $\mult{\lambda}{1} < k$. Thus $\mult{\lambda}{1} = k-1$. Additionally, for any prime $q$, we have $m_\lambda(q)=m_\nu(q)$. From here, we can follow the argument in the proof of Lemma~\ref{lem:injectivity on partitions with 1s}.
\end{proof}

Theorem~\ref{th:injectivity}, now, follows from Lemmas~\ref{lem:injectivity on partitions with 1s} and~\ref{lem:injectivity on partitions with no 1s but a prime part}.

As stated before, the injectivity of $\pre_k: \injclass \rightarrow \cP$ is not a subcase of Conjecture~\ref{prek:con} because it makes no assumptions about the size of $\lambda$ in $\injclass$. Indeed, injectivity is not true in general for $\pre_k : \cP \rightarrow \cP$. For example, $\pre_2(4,1) = \pre_2(2,2) = (4)$, but of course $|(4,1)| = 5$ while $|(2,2)| = 4$. Thus injectivity of $\pre_k$ on $\injclass$ and on 
$\cP(n)$ are related, but distinct, properties.

\begin{remark}
     The case $k=2$ is special. Indeed, when $\nu=\pre_2(\lambda)$ for $\lambda\in \injtwo$, there are exact recursive formulas for the multiplicities of the parts of $\lambda$ in terms of the multiplicities of the parts of $\nu$. 
     
     \begin{itemize}
         \item If $ \mult{\nu}{1} >0$, then Equation~\eqref{eqn:binomial for counting 1s}  can be rewritten as a quadratic equation with  exactly one positive solution:
        \begin{equation}\label{eqn:counting 1s}
         \mult{\lambda}{1} = \frac{1+\sqrt{1+8\mult{\nu}{1}}}{2}
        \end{equation}
        and for any  $y>1$, as $\mult{\lambda}{x}$ is  determined by $\nu$ for all $x<y$, we have  
        \begin{equation}\label{eqn: counting any y}
        \mult{\lambda}{y}=\frac{1}{\mult{\lambda}{1}}\left(\mult{\nu}{y}-\left( \sum_{\substack{d\mid y \\ 1<d<\sqrt{y}}} \mult{\lambda}{d}\mult{\lambda}{y/d} \right) -
        \binom{\mult{\lambda}{\sqrt{y}}}{2}\right),\end{equation} where $\mult{\lambda}{q} = 0$ when $q$ is not an integer.  

        \item If $\mult{\nu}{1}=0$ and $\mult{\nu}{p}>0$ for some prime $p$, then $\mult{\lambda}{1}=1$ and if $y>1$, Equation~\eqref{eqn: counting any y} determines $\mult{\lambda}{y}$ recursively. 
    \end{itemize}
\end{remark}

In~\cite{Li}, Li proved Conjecture~\ref{prek:con} for $k=2$. Li's work, proved independently of our own efforts, uses an inductive argument to establish the uniqueness of the parts of $\lambda$, in decreasing order of size. Note that \cite{Li} only determines uniqueness of the preimage under $\pre_2$, and does not actually determine the parts of that preimage. 

\begin{theorem}[{\cite[Theorem 1.2]{Li}}] \label{th:Li} 
    The map $\pre_2 : \cP(n) \rightarrow \cP$ is injective. 
    \end{theorem}

\section{Consequences of Theorem~\ref{th:injectivity}}\label{sec:consequences of our injectivity}

In this section we prove several results about the total number of parts of particular fixed size ($1$ or a prime) in the  image under $\pre_k$ of certain sets of partitions of $n$.  These results were initially conjectured by Beck, some appearing in~\cite{bbms} and other communicated to us privately. In~\cite{bbm, bbms},  when only the injectivity of 
$\pre_2:\cB(n)\to \cB$
had been proved, the authors only looked at the number of parts equal to $1$, $2$, and $4$, respectively, in the image of $\cB(n)$ under $\pre_2$. However, Theorem~\ref{th:injectivity} allows us to consider the total number of parts equal to $1$, or to a prime $p$, among all partitions in 
$\pre_k(\cP(n))$, the image of $\cP(n)$ under $\pre_k$. 
More precisely, the injectivity of $\pre_k$ on $\injclass$ means that we can use information about $\lambda \in \injclass$ to infer, uniquely, results about $\pre_k(\lambda)$. We will use such inferences implicitly in the arguments of this section.

Let $n$ and $k$ be positive integers.  Define
$$a_{i,k}(n) := \hspace{-.1in} 
\sum_{\nu \in \pre_k(\cP(n))} \hspace{-.1in} \mult{\nu}{i}$$
to be the total number of parts equal to $i$ in all partitions in 
$\pre_k(\cP(n))$. 
Some of our arguments will benefit from assigning colors to parts of a partition. As such, for positive integers $n$ and $k$, let $\colorclass{k}{n}$ be the set of partitions of $n$ in which part $1$ occurs in $k$ different colors, $1_1, \ldots, 1_k$, and all other parts are uncolored, and set $\colorclasscount{k}{n}:=|\colorclass{k}{n}|$.  We order the colors so that $1_1>\cdots >1_k$. For example, $\colorclasscount{2}{3}=7$ because 
$$\colorclass{2}{3}=\left\{(3), \, (2, 1_1), \, (2,1_2),\, (1_1,1_1,1_1), \, (1_1,1_1,1_2), \, (1_1,1_2,1_2), \, (1_2,1_2,1_2)\right\}.$$
Moreover, let $\twocolorclass{k}{p}{n}$ be the set of partitions of $n$ in which part $1$ can occur in $k$ different colors as above, part $p$ can occur in two colors, $p_1>p_2$, and all other parts are uncolored. Set $\twocolorclasscount{k}{p}{n}:=|\twocolorclass{k}{p}{n}|$.  For example, $\twocolorclasscount{2}{2}{3}=9$ as
\begin{align*}
\twocolorclass{2}{2}{3}=\left\{(3), \, (2_1, 1_1), \, (2_2,1_1), \right.& \left. \!(2_1,1_2), \, (2_2,1_2), \, (1_1,1_1,1_1), \right. \\ 
& \left. (1_1,1_1,1_2),  \, (1_1,1_2,1_2),  \, (1_2,1_2,1_2)\right\}.\end{align*}
 
Some of these enumerations have been studied previously. For example, the sequence $\colorclasscount{3}{n}$ is~\cite[A014153]{OEIS}, while $\twocolorclasscount{2}{2}{n}$ is~\cite[A000097]{OEIS}.

The next theorem shows that the total numbers of $1$ and of primes $p$ in the image of $\pre_k$ is equal to the number of certain families of colored partitions. We include bijective arguments here, but can also be proved via counting arguments.

\begin{theorem}\label{thm:conj_parts_k} Let $n,k$ be positive integers and let $p$ be a prime.  Then  \begin{enumerate}[label={\textup{(\alph*)}}]\setlength{\itemsep}{.1in}
    \item $a_{1,k}(n)=\colorclasscount{k+1}{n-k}$, and 
    \item  $a_{p,k}(n)=\twocolorclasscount{k}{p}{n-p-k+1}$. 
\end{enumerate}
\end{theorem}

\begin{proof} 

\begin{enumerate}[label=(\alph*)]\setlength{\itemsep}{.1in} 

\item 
As in~\cite{S}, a rooted partition is a partition with one or more parts distinguished as \emph{roots}. For example $(5, 5, 3, \hat 3, 3, 1, 1, \hat 1)$ and $(5, 5, \hat 3,  3, 3, 1, \hat 1, 1)$ are two different rooted partitions of $22$, each with two roots. The partition $\nu=\pre_k(\lambda)$ has a part equal to $1$ if and only if $\mult{\lambda}{1} \ge k$. 
A part equal to $1$ in $\nu=\pre_k(\lambda)$ is obtained from a $k$-tuple  $(\lambda_{i_1}, \ldots,\lambda_{i_k})$ 
of distinct parts  all equal to $1$.
Thus, if $\cR_{k}(n)$ is the set of rooted partitions of $n$ with $k$ roots all equal to $1$, then $a_{1,k}(n)=|\cR_{k}(n)|$.  We create a bijection from $\cR_{k}(n)$ to $\colorclass{k+1}{n-k}$. If $\lambda \in \cR_{k}(n)$ has rooted  parts $\lambda_{i_j}=1$ for $1\leq j\leq k$, we replace 
 all parts of $\lambda$ equal to $1$ that are strictly before $\lambda_{i_1}$ by $1_1$; for $2\leq j\leq k$, we replace all parts strictly between $\lambda_{i_{j-1}}$ and $\lambda_{i_j}$ by $1_j$; and  all parts after $\lambda_{i_k}$ by $1_{k+1}$. Now remove the (uncolored) roots, obtaining a partition  $\mu\in \colorclass{k+1}{n-k}$.  This transformation is clearly reversible: given $\mu\in \colorclass{k+1}{n-k}$, for each $1\leq j\leq k$, insert a rooted $\hat 1$ after the last occurrence of $1_j$, and uncolor all colored parts to obtain $\lambda\in \cR_{k}(n)$. This establishes a bijection between $\cR_{k}(n)$ and $\colorclass{k+1}{n-k}$, and completes the proof. 

\item 

Fix $\lambda \in \cP(n)$. Then $\nu = \pre_k(\lambda)$ has a prime part $p$ if and only if $\lambda$ has a $k$-tuple of different parts $(p,1,1,\ldots,1)$. 
Thus, if $\cR^p_{k}(n)$ is the set of rooted partitions of $n$ with $k$ roots $\hat p, \hat 1, \hat 1, \ldots, \hat 1$, then $a^p_{1,k}(n)=|\cR^p_{k}(n)|$.  We create a bijection from $\cR^p_{k}(n)$ to $\colorclass{k+1}{n-p-k+1}$. If $\lambda \in \cR^p_{k}(n)$ has rooted parts $\lambda_{i_1}=p$, $\lambda_{i_j}=1$, $2\leq j\leq k$, we replace 
all parts of $\lambda$ equal to $p$ before $\lambda_{i_1}$ by $p_1$ and all parts of $\lambda$ equal to $p$ after $\lambda_{i_1}$ by $p_2$; the coloring of unrooted parts equal to $1$ is similar to that in the proof of part (a).  Now remove the (uncolored) roots, obtaining a partition $\mu\in \twocolorclass{k}{p}{n-p-k+1}$. As in the proof of part (a),  this transformation is clearly reversible. 
This establishes a bijection between $\cR^p_{k}(n)$ and $\twocolorclass{k+1}{p}{n-p-k+1}$, and completes the proof. 
\end{enumerate} 
\end{proof}

The generating functions for the sequences $\colorclasscount{k+1}{n}$ and $\twocolorclasscount{k}{p}{n}$  are easily derived from the definitions of these sequences. Then Theorem~\ref{thm:conj_parts_k} leads to the generating functions for the sequences $a_{1,k}(n)$ and $a_{p,k}(n)$, generalizing~\cite[Theorems~4(a) and~6(a)]{bbms}.

\begin{corollary} \label{gf_a_p,k} For $k\geq 1$ and $p$ either $1$ or a prime, we have the generating function
$$\sum_{n\geq 0}a_{p,k}(n)q^n = \frac{q^{p+k-1}}{(1-q)^{k-1}(1-q^p)(q;q)_\infty}.$$
    \end{corollary}

Next, we generalize~\cite[Theorems 4(b) and 6(b)]{bbms}.

\begin{theorem} 
If $k\geq 2, n\geq 0$, and $p$ is a prime, then 
\begin{align*}a_{1,k}(n)& =\sum_{i\geq 1}\binom{i-1}{k-1}\part(n-i), \text{ and}\\ a_{p,k}(n)& =\sum_{i, j\geq 1}\binom{i-1}{k-2}\part(n-i-pj),\end{align*}
where $\part$ is the partition function defined in Section~\ref{sec:intro}.
    \end{theorem}
    
\begin{proof} The proofs are similar to those of~\cite[Theorems 4(b)]{bbms}. Recall from the proof of Theorem~\ref{thm:conj_parts_k} that $$a_{p,k}(n)=\begin{cases} |\cR_{k}(n)| & \text{ if } p=1, \text{ and} \\ |\cR^p_{k}(n)| & \text{ if } p \text{ is a prime}.
    \end{cases}$$
A rooted partition $\lambda \in \cR_{k}(n)$ corresponds to a pair of partitions $(\alpha, \beta)$, where $\alpha$ consists of the parts of $\lambda$ before the first $\hat 1$
 and $\beta$ consists of the parts of $\lambda$  after the first $\hat 1$ with hats removed. The first $\hat 1$ does not belong to either partition. If the first $\hat 1$ is in the $i$th entry from the right in $\lambda$, then $\alpha \in \cP(n-i)$ and $\beta$ consists of $i-1$ parts equal to $1$. For each such pair, there are $\binom{i-1}{k-1}$ ways to choose $k-1$ roots in $\beta$. Summing over $i$ completes the proof of the first identity. The second identity is proved similarly with the partition $\beta$ consisting of all parts after the first $\hat 1$ and all parts equal to $p$ after $\hat p$. The first $\hat 1$ and $\hat p$ are in neither $\alpha$ nor $\beta$, and we assume that the first $\hat p$ is followed by $j-1$ parts $p$. 
\end{proof}

The authors of~\cite{bbms} also prove results involving forward differences of sequences. 
For a sequence $s(n)$, we denote its forward difference by 
$$\Delta s(n):=s(n+1)-s(n).$$
For example, the forward differences $\Delta a_{1,2}(n)$, $\Delta a_{2,2}(n)$, and $\Delta a_{3,2}(n)$  are sequences A000070, A024786, A024787 in~\cite{OEIS}, respectively.

In general, we have the following result for $\Delta a_{p,k}(n)$. As with Theorem~\ref{thm:conj_parts_k}, we provide a bijective proof, but we can also prove the result analytically using generating functions.

\begin{theorem}\label{fwd}
For $k\geq 2$ and  $p$ equal to $1$ or a prime, we have 
\begin{equation}\label{fwd1} \Delta a_{p,k}(n)=a_{p,k-1}(n).\end{equation}
\end{theorem}

\begin{proof} 
Once again, we use rooted partitions. 
Throughout the proof, the variable $p$ is either $1$ or a prime, and we maintain the notation for rooted partitions from the proof of Theorem~\ref{thm:conj_parts_k}. For the duration of this proof, write $\widetilde{\cR}^p_k(n)$ to mean $\cR_{k}(n)$ if $p=1$, and $\cR^p_{k}(n)$ if $p$ is prime. 

We want to interpret the difference $a_{p,k}(n+1)-a_{p,k}(n)$. 
We have a bijection between $\widetilde{\cR}^p_{k}(n)$ and the subset of rooted partitions in $\widetilde{\cR}^p_{k}(n+1)$ with last part unrooted, defined by appending an unrooted $1$ to the end of an element of $\widetilde{\cR}^p_{k}(n)$. Thus, $a_{p,k}(n+1)-a_{p,k}(n)$ equals the number of rooted partitions in $\widetilde{\cR}^p_{k}(n+1)$ with last part rooted. These partitions are in bijection with  $\widetilde{\cR}^p_{k-1}(n)$, the bijection being to remove last (rooted $\hat 1$) part. Clearly, this transformation is reversible and $\Delta a_{p,k}(n)=|\widetilde{\cR}^p_{p,k-1}(n)|$, completing proof. 
\end{proof}

As in Definition~\ref{defn:multiplicity}, let 
$\mult{n}{i}$
denote the total number of parts equal to $i$ among all partitions of $n$. That is,
$$\mult{n}{i} = \sum_{\lambda\in \cP(n)} \mult{\lambda}{i}.$$
With this statistic, we single out the case $k=2$ of Theorem~\ref{fwd}. 

\begin{corollary}\label{conj_parts_2} 
If $p$ is equal to $1$ or a prime, then the forward difference $\Delta a_{p,2}(n)$ is equal to $\mult{n}{p}$. Moreover, the generating function for $\mult{n}{p}$ is 
$$\sum_{n\geq 0}\mult{n}{p}q^n=\frac{q^p}{(1-q^p)(q;q)_\infty}.$$
    \end{corollary}

Iterating~\eqref{fwd1}  proves the last part of \cite[Conjecture 18]{bbms}. 

\begin{theorem} Let $k\geq 2$, $n\geq 0$. If $p$ is equal to $1$ or a prime, then
$$\Delta^{k-1}a_{p,k}(n) = \mult{n}{p}.$$
\end{theorem}

We note that the results proved so far in this section are also valid if we replace the set of all partitions of $n$ by any subset of partitions of $n$, that is if we investigate the number of parts equal to $1$ or a prime $p$ in the image under $\pre_k$ of a subset of partitions of $n$.  

 We conclude this section by proving generalizations of  two conjectures of Beck, which were communicated to us privately.  To do so, we introduce some notation: we write $\cO(n)$ (respectively, $\cD(n)$) for the set of all partitions of $n$ with odd (respectively, distinct) parts.

\begin{definition}
    Let $z_k(n)$ be the total number of different parts that occur in  partitions in $\pre_k(\cO(n))$.
    That is,
    $$z_k(n) = \#\{x : x = a_1\cdots a_k, \text{ where } a_i, 1\leq i\leq k\text{ are parts in some odd partition $\lambda$ of } n\}.$$ 
\end{definition}

 \begin{theorem}\label{thm:z(2n) = z(2n-1)}
 For any positive integer $n$, and an odd positive integer $k$ we have $$z_k(2n)=z_k(2n-1).$$ 
    \end{theorem}
    
\begin{proof} 
First notice that, for any $m\geq 0$,  if $s$ is a part of $\pre_k(\lambda)$ for some $\lambda\in \cO(m)$ then $s$ is also a part of $\pre_k(\lambda\cup(1))$, where $\lambda\cup(1)$ is the partition of $m+1$ obtained from $\lambda$ by appending a part equal to $1$. Thus, $z_k(m)$ is an increasing function of $m$, and $z_k(2n-1)\leq z_k(2n)$. To prove the reverse inequality, suppose $s$ is a part of $\pre_k(\mu)$ for some $\mu\in \cO(2n)$. Hence $\mu$ has $k$ parts $\mu_{i_1}, \ldots, \mu_{i_k}$ with $i_1<\cdots <i_k$ and 
$\mu_{i_1}\cdots \mu_{i_k}=s$. Since $2n$ is even and $\{k,\mu_{i_1},\ldots,\mu_{i_k}\}$ are all odd, the partition $\mu$ must actually have at least one other part. Choose any part $\mu_t$ with $t\neq i_j$, $1\leq j\leq k$.  We create a partition $\mu^*\in\cO(2n-1)$ from $\mu\in \cO(2n)$ by removing the part $\mu_t$ and inserting $\mu_t-1$ parts equal to $1$ (if $\mu_t=1$, we just remove part $\mu_t$). Then $s$ is a part of $\pre_k(\mu^*)$. Hence $z_k(2n)\leq z_k(2n-1)$, and so $z_k(2n) = z_k(2n-1)$.
    \end{proof}

The partition $\mu^*$ defined in the proof of Theorem~\ref{thm:z(2n) = z(2n-1)} does not necessarily have a unique preimage $\mu$, but this is irrelevant to the evaluation of the function $z$. Indeed, the proof did not use the injectivity of the map $\pre_k$ on $\cO(n)$.

Let $p$ be an odd prime and let $o_{2,p}(n)$ be the total number of parts equal to $p$ among all partitions in the image $\pre_2(\cO(n))$, and let $d_{2,p}(n)$ be the total number of parts equal to $p$ among all partitions in the image $\pre_2(\cD(n))$. Beck conjectured the result below, which we prove here.

\begin{theorem} 
 For $n\geq 0$ and $p$ an odd prime, we have $$o_{2,p}(n)\geq d_{2,p}(n)$$ and $$o_{2,p}(n)\equiv d_{2,p}(n)\pmod 2.$$ 
\end{theorem}

\begin{proof} 
Glaisher's bijection (see, for example,~\cite{AE}) is used to prove that $|\cO(n)|=|\cD(n)|$.
    Under that bijection, a  partition $\lambda\in \cD(n)$  that has parts equal to both $1$ and $p$ maps to a partition $\mu\in \cO(n)$  with $\mult{\mu}{1}$ and $\mult{\mu}{p}$ both odd. This proves both assertions. 
\end{proof}

There is also a consequence of \cite{Li} related to the sets $\cO(n)$ and $\cD(n)$. As shown in \cite{bbm}, it follows from  the injectivity of $\pre_2$ that the cardinality of $\cO \cap \pre_2(\cP(n))$ is at least as large as the cardinality of $\cD \cap \pre_2(\cP(n))$. This is in contrast to Euler's identity that states that the number of odd partitions in $\pre_1(\cP(n))$ equals the number of distinct partitions in $\pre_1(\cP(n))$.

\section{Consequences of the full injectivity}\label{sec:consequences of full injectivity}

Recall that the plethysm $f[g]$ of symmetric functions is defined by substituting the monomials of $g$ into the variables of $f$. More precisely, if $g= \sum_{a} x^{a^i}$ where each $a^i$ is a degree vector and repeated monomials appear multiple times, then $f[g] := f(x^{a^1},x^{a^2},\dots). $

A consequence of the injectivity from \cite{Li} is that the family of plethysms $\{e_r[e_2]\}_{r \ge 1}$ determines the original elementary symmetric functions. 
Concretely, after specializing to finitely many variables $(x_1,\dots,x_\ell)$, each $e_r[e_2]$ is the $r$th elementary symmetric polynomial in the $\binom{\ell}{2}$ quadratic monomials $\{x_ix_j : i<j\}$. Thus, the full family $\{e_r[e_2]\}$ encodes the elementary symmetric functions of these quadratic monomials. From this data, one can recover the underlying multiset $\{x_ix_j : i<j\}$, and by Li's injectivity result, this multiset uniquely determines the variables $(x_1,\dots,x_\ell)$. Consequently, the family $\{e_r(x_1,\dots,x_\ell)\}$ can be reconstructed. While plethysm is, in general, not invertible, this special case
retains enough structure to recover the entire family $\{e_r\}_{r\ge1}$. Furthermore, assuming Conjecture \ref{prek:con}, that the map $\pre_k: 
\cP(n)\to\cP$ is injective, the same conclusion would follow for general $k$, meaning that the family of plethysms $\{e_r[e_k]\}$ evaluated at the parts of a partition $\lambda$ would determine the entire family $\{e_r\}$ evaluated at $\lambda$.

Put another way, if $f(x)$ is a monic degree-$d$ polynomial with roots $x_1,\ldots, x_d$, then Vieta's theorem says that the coefficient of $x^{d-r}$ is $\pm e_r(x_1,\ldots, x_d)$ where $e_r$ is the elementary symmetric polynomial. 
Thus if we know $\nu$ in the image of $\pre_2$ (and, hence, its size), and we know the size $n$ of a preimage of $\nu$, then we would know the coefficients of $x^{n-1}$ and $x^{n-2}$ in the monic polynomial whose roots are the parts of that preimage. I.e., the polynomial $\prod (x-\lambda_i)$. Li's result says that $\nu$ determines its preimage $\lambda$, which would give the rest of $f$'s coefficients, as well.

\section{Open questions}\label{sec:open}

The work of this paper, and in the papers cited here, suggests several directions for future research. We mention a few of these here.

The most obvious of these is still the resolution of Conjecture~\ref{prek:con}. While Theorems~\ref{th:injectivity} and~\ref{th:Li} make important and substantial strides towards establishing injectivity of $\pre_k$, the original conjecture itself still stands. The problem seems like it could be susceptible to several different approaches, and indeed the proofs of Theorems~\ref{th:injectivity} and~\ref{th:Li} already use distinct strategies themselves.

Much of this work has been concerned with the behavior of the map $\pre_k$, but we might also consider its image. That is, what exactly is the set of partitions $\pre_k(\cP(n))$? Or, $\pre_k(\injclass)$, or even $\pre_2(\cP(n))$?

\section*{Acknowledgments}

The authors are grateful to the Institute for Advanced Study and the organizers of the June 2025 IAS Research Community in Algebraic Combinatorics, at which this collaboration and project began.

Bridget Tenner's research is partially supported by NSF Grant DMS-2054436.

Part of this paper is based upon work supported by NSF grant DMS-1929284, while Chenchen Zhao was in residence at the Institute for Computational and Experimental Research in Mathematics in Providence, RI, during the Categorification and Computation in Algebraic Combinatorics semester program.

\end{document}